\documentclass[11pt]{article}

\usepackage[margin=1.15in]{geometry}
\usepackage{amsmath,amssymb,amsthm,mathtools}
\usepackage{booktabs}
\usepackage{enumitem}
\usepackage{microtype}
\usepackage[hidelinks]{hyperref}
\usepackage{xurl}

\newtheorem{theorem}{Theorem}[section]
\newtheorem{proposition}[theorem]{Proposition}
\newtheorem{lemma}[theorem]{Lemma}
\newtheorem{corollary}[theorem]{Corollary}
\newtheorem{remark}[theorem]{Remark}
\newtheorem{definition}[theorem]{Definition}

\newcommand{\Prob}{\mathbb P}

\newcommand{\Z}{\mathbb Z}

\newcommand{\eps}{\varepsilon}
\newcommand{\wc}{\mathrm{wc}}
\newcommand{\prodprob}{\mathrm{prod}}
\newcommand{\EMPTY}{\bot}
\newcommand{\score}{\mathsf S}

\newcommand{\StackableAt}{\mathsf{StackableAt}}

\title{A Fixed-Offset Transition for Random Stackability on Paths}
\author{John Fairfax-Ball}
\date{September 2026}

\begin{document}
\maketitle

\begin{abstract}
We study a support-collapse version of graph pebbling on the path.  A
configuration is \emph{stackable} if a sequence of legal pebbling moves can
produce a nonzero configuration supported on a single vertex.  On the path
$P_n$, choose a configuration uniformly from all weak compositions of the
fixed total $n\mu_n$, where $\mu_n$ is a positive integer.  Set
\[
 c_n=\sqrt{\log_2 n}-\frac12\log_2\log_2 n+\log_2(3e).
\]
We prove a two-sided fixed-offset transition: for every fixed $\eps>0$, the
stackability probability tends to $0$ if
$\log_2\mu_n\le c_n-\eps$ eventually, and tends to $1$ if
$\log_2\mu_n\ge c_n+\eps$ eventually.  No assertion is made at zero offset.
The proof starts from an exact recursive stackability score on trees, reduces
path messages to a one-dimensional integer recurrence, and identifies rare
dyadic deficit excursions.  Classical binary-partition asymptotics give the
one-sided excursion rate
\[
 -\log_2 q_\mu
 =x^2+2x\log_2x-2\log_2(3e)x+o(x),\qquad x=\log_2\mu.
\]
A constant-cost regeneration step turns this local estimate into a spatial
front mechanism, while an exact deep-message necessity theorem controls the
opposite side.  Conditioning independent geometric occupancies on their sum
then returns the uniform fixed-total model.  The finite deterministic
necessity theorem and its fixed-total corollary are formalised in Lean and are
registered with Palomar; the asymptotic transition theorem itself is not part
of that registration.
\end{abstract}

\section{Introduction}

A pebbling move on a graph removes two pebbles from one vertex and places one
pebble on an adjacent vertex.  Because every move dissipates one pebble, the
question of whether a configuration can be concentrated onto a single vertex
has a strong dependence on geometry.  Csern\'ak and Soukup use
\emph{stackable} for precisely this support-collapse event and study the
corresponding deterministic worst-case stacking number~\cite{CsernakSoukup2026}.
The present paper asks a different question: how likely is a uniformly random
fixed-total configuration to be stackable?

The probability space is the uniform multiset, or weak-composition, model that
is standard in random graph pebbling.  It is the same allocation law used in
the classical threshold literature~\cite{CzygrinowEtAl2002,BekmetjevEtAl2003}
and in the recent sharp work of Bushaw and Kettle~\cite{BushawKettle2025}.
The event, however, is different.  Classical random pebbling asks for
solvability at prescribed roots (or at every root), whereas stackability asks
whether the entire surviving support can be collapsed to \emph{some} one
vertex.  This distinction matters even on paths.

Random pebbling on paths already exhibits a stretched-logarithmic scale.
Moews obtained the order
$n2^{\sqrt{\log_2 n}}/\sqrt{\log_2 n}$ for the path solvability threshold
\cite{Moews2019}, following earlier path bounds including
\cite{WiermanEtAl2004,CzygrinowHurlbert2008}.  Bushaw--Kettle subsequently
sharpened the path analysis in the same fixed-total law using partitions into
powers~\cite{BushawKettle2025}.  Our stackability transition has the same
broad scale but a different deterministic mechanism.  At density level, its centre is
\[
 2^{c_n}=\frac{3e}{\sqrt{\log_2 n}}\,2^{\sqrt{\log_2 n}}.
\]
The additive constant $\log_2(3e)$ comes from a sharp local message-excursion
calculation rather than from a monotonicity argument.

There is a second reason to avoid importing the usual threshold formalism
uncritically.  The finite stackability probability need not be monotone in the
total number of pebbles.  For example, on $P_2$ the probabilities at totals
$1,2,3$ are respectively $1,2/3,1$.  We therefore refer to the result as a
\emph{high-density stackability transition}; the theorem is a two-sided
fixed-offset statement and does not require finite-$n$ monotonicity.

The deterministic starting point is an exact tree certificate developed in
the companion TreeStack project~\cite{TreeStack2026}.  For a finite tree $T$,
a configuration $C$, and a proposed root $r$, TreeStack recursively computes
integer branch messages and a rooted score $\score_r(C)$, with
\[
 \StackableAt(T,C,r)\quad\Longleftrightarrow\quad \score_r(C)>0.
 \tag{1.1}\label{eq:treestack-criterion}
\]
The present paper uses this result as an input; it is not a new ProbStack
theorem.  On a path, the recursion becomes one-dimensional and exposes a
low-phase deficit that doubles unless replenished.  Rare multiscale failures of
replenishment produce irreversible fronts.

The proof has four main components.  First, the path recursion is converted
into an exact dyadic budget.  Second, a de Bruijn binary-partition asymptotic
\cite{deBruijn1948} yields the logarithmic probability of a local excursion,
including the linear term that determines the constant $3e$.  Third, a
constant-cost regeneration step makes this local rate uniform over incoming
states and allows repeated spatial trials.  Fourth, an exact deterministic
necessity theorem shows that every nonstackable fixed-total configuration must
contain a directed message at most $-(2\mu-1)$.  The subcritical side is then
proved by abundance of certified fronts, and the supercritical side by a union
cover of all possible deep messages.

The fixed-total law is handled exactly.  Independent geometric occupancies of
mean $\mu$, conditioned on their sum being $n\mu$, are uniform over weak
compositions.  This classical allocation identity is part of the general
conditioned-allocation framework surveyed by Janson~\cite{Janson2012}.  We use
it together with finite likelihood-ratio estimates; in particular, we never
assume independence after conditioning.

A public-record literature audit carried out for the project found no exact
match, among the sources searched, for the combination of the stackability
event, the uniform fixed-total law, paths, and the fixed-offset theorem below.
This is a search-status statement, not a priority claim.  The closest lines of
work include the deterministic stackability problem of Csern\'ak--Soukup
\cite{CsernakSoukup2026}, random path solvability in the same allocation law
\cite{Moews2019,BushawKettle2025}, classical binary partitions
\cite{Mahler1940,deBruijn1948}, and signed resource recurrences in tree cover
pebbling such as Watson's leaf reduction~\cite{Watson2005}.  Related tree
target-pebbling methods appear in~\cite{AdautoEtAl2026}.  Cover pebbling
uses ``stacking'' in a different extremal sense; for example, Sj\"ostrand's
cover-pebbling theorem concerns initially stacked extremal configurations for a
cover demand rather than the support-collapse event studied here
\cite{Sjoestrand2005}.

\section{Model and main result}\label{sec:model}

Let $G=(V,E)$ be a finite simple graph.  A \emph{configuration} is a function
$C:V\to\Z_{\ge0}$.  A legal pebbling move from $u$ to an adjacent vertex $v$
requires $C(u)\ge2$ and replaces $C(u)$ by $C(u)-2$ and $C(v)$ by $C(v)+1$.
The total mass $|C|=\sum_{v\in V}C(v)$ therefore decreases by one at every
move.

\begin{definition}[Stackability]
A configuration is \emph{stacked at $r$} if $C(r)>0$ and $C(v)=0$ for every
$v\ne r$.  It is \emph{stackable at $r$} if a finite sequence of legal moves
reaches a configuration stacked at $r$, and is \emph{stackable} if it is
stackable at at least one vertex.
\end{definition}

Write $P_n$ for the path with vertices $0,1,\dots,n-1$.  For integers $n\ge1$
and $t\ge0$, let
\[
 \mathcal D_{n,t}
 =\left\{c\in\Z_{\ge0}^n:\sum_{i=0}^{n-1}c_i=t\right\}
\]
with the uniform distribution.  Thus
\[
 |\mathcal D_{n,t}|=\binom{n+t-1}{t}.
\]
This is the uniform multiset (Bose--Einstein) model, not the multinomial law
obtained by independently placing labelled pebbles.

The main theorem concerns totals $t_n=n\mu_n$ with integer mean occupancy
$\mu_n$.

\begin{theorem}[Fixed-offset stackability transition]\label{thm:main}
Let $(\mu_n)$ be a sequence of positive integers, let $C_n$ be uniform on
$\mathcal D_{n,n\mu_n}$, and put
\[
 c_n=\sqrt{\log_2 n}-\frac12\log_2\log_2 n+\log_2(3e).
\]
For every fixed $\eps>0$ the following hold as $n\to\infty$:
\begin{align*}
 \log_2\mu_n\le c_n-\eps\ \text{eventually}
 &\quad\Longrightarrow\quad
 \Prob(C_n\text{ is stackable})\longrightarrow0,\\
 \log_2\mu_n\ge c_n+\eps\ \text{eventually}
 &\quad\Longrightarrow\quad
 \Prob(C_n\text{ is stackable})\longrightarrow1.
\end{align*}
No assertion is made when the offset tends to $0$.
\end{theorem}

The theorem is best read on the logarithmic density scale.  If
$N=\log_2 n$ and $s=\sqrt N$, then the centre is
$c_n=s-\log_2s+\log_2(3e)$.  The proof will identify a local rare-event rate
$R(x)$ satisfying
\[
 R(x)=x^2+2x\log_2x-2\log_2(3e)x,
 \tag{2.1}\label{eq:R}
\]
and show that a fixed displacement $\delta$ from the centre gives
\[
 N-R(c_n+\delta)=-2\delta\sqrt N+o(\sqrt N).
 \tag{2.2}\label{eq:balance-preview}
\]
Thus a fixed negative offset creates exponentially many effective spatial
opportunities for a bad front, whereas a fixed positive offset makes the
union of all deep-message obstructions vanish.

\begin{remark}[Finite nonmonotonicity]
On $P_2$, the exact stackability probabilities at totals $1,2,3$ are
$1,2/3,1$.  On $P_3$, the probability decreases from $19/21$ at total $5$ to
$25/28$ at total $6$, before becoming $1$ at total $7$.  These finite facts
are why Theorem~\ref{thm:main} is stated directly rather than deduced from a
monotone-family threshold theorem.
\end{remark}

\section{Exact path messages and deterministic fronts}\label{sec:messages}

We now specialize the rooted-score criterion~\eqref{eq:treestack-criterion} to
a path.  The integer transfer function is
\[
 F(y)=
 \begin{cases}
  2y-3,&y\le1,\\
  1,&y=2,\\
  0,&y=3,\\
  y/2,&y\ge4\text{ even},\\
  (y-3)/2,&y\ge5\text{ odd}.
 \end{cases}
 \tag{3.1}\label{eq:F}
\]
Branch messages have one additional categorical state, denoted $\EMPTY$,
meaning that the branch contains no pebbles.  It is important that
$\EMPTY$ is not the integer $0$.  For numerical sums only, write
$[\EMPTY]=0$ and $[m]=m$ for $m\in\Z$.

For $x\in\Z_{\ge0}$ define the path update
\[
 \Phi(\EMPTY,x)=
 \begin{cases}
 \EMPTY,&x=0,\\
 F(x),&x>0,
 \end{cases}
 \qquad
 \Phi(m,x)=F(m+x)\quad(m\in\Z).
 \tag{3.2}\label{eq:path-step}
\]
For a configuration $C=(C_0,\dots,C_{n-1})$, the left-to-right outgoing
messages are
\[
 \ell_0=\Phi(\EMPTY,C_0),\qquad
 \ell_i=\Phi(\ell_{i-1},C_i)\quad(1\le i\le n-2),
 \tag{3.3}\label{eq:left-scan}
\]
and the right-to-left messages $\rho_i$ are defined by the same recurrence on
the reversed path.  The message entering vertex $i$ from the left is
$\ell_{i-1}$ when $i>0$, and the one entering from the right is $\rho_{i+1}$
when $i<n-1$.

\begin{proposition}[Path rooted score]\label{prop:path-score}
For every $0\le i<n$,
\[
 \score_i(C)
 =C_i+\mathbf1_{i>0}[\ell_{i-1}]
      +\mathbf1_{i<n-1}[\rho_{i+1}],
 \tag{3.4}\label{eq:path-score}
\]
and
\[
 \StackableAt(P_n,C,i)\quad\Longleftrightarrow\quad \score_i(C)>0.
\]
Consequently $C$ is stackable if and only if
$\max_i\score_i(C)>0$.
\end{proposition}

\begin{proof}
The path branches incident to $i$ are exactly its nonempty left and right
components.  Applying the TreeStack branch recursion and rooted-score theorem
to those two branches gives~\eqref{eq:path-score}; an empty branch contributes
zero to the score but remains categorically distinct from an integer-zero
message.  The equivalence is then exactly~\eqref{eq:treestack-criterion}.
\end{proof}

Once a scan has become active, it is the scalar recurrence
$M_j=F(M_{j-1}+C_j)$.  Its rare failures are controlled by the low branch of
$F$.

\begin{lemma}[Dyadic deficit identity]\label{lem:deficit}
Suppose $M$ is an active integer message and the next occupancy is $c$, with
$c+M\le1$.  Put $Z=3-M$.  Then after the update
\[
 Z'=2(Z-c).
 \tag{3.5}\label{eq:deficit-step}
\]
If a block $c_1,\dots,c_k$ remains entirely in this low phase, then
\[
 Z_k=2^kZ_0-\sum_{j=1}^k2^{k-j+1}c_j.
 \tag{3.6}\label{eq:deficit-block}
\]
\end{lemma}

\begin{proof}
In the low phase, $M'=2(c+M)-3$ by~\eqref{eq:F}, so
$3-M'=2(3-M-c)$.  Iterating gives~\eqref{eq:deficit-block}.
\end{proof}

The next deterministic observation turns a sufficiently negative message into
a spatial exclusion certificate.

\begin{lemma}[Message bounded by branch mass]\label{lem:message-mass}
If a nonempty path branch has total mass $A$ and integer outgoing message $m$,
then $m\le A$.
\end{lemma}

\begin{proof}
The transfer satisfies $F(y)\le y$ for every integer $y$.  The one-vertex case
is immediate.  Inductively, if the child message is at most the child mass,
then the next message is at most the occupancy at the new vertex plus that
mass.
\end{proof}

\begin{proposition}[Irreversible exclusion front]\label{prop:front}
Consider a left-to-right active message $m_i$ sent across the cut after vertex
$i$, and let
$A_i=\sum_{j=i+1}^{n-1}C_j$ be all mass to its right.  If
\[
 m_i+A_i\le0,
 \tag{3.7}\label{eq:front-condition}
\]
then no vertex strictly to the right of the cut can be a stacking root.  The
symmetric statement holds from the right.  In particular, if a left exclusion
region and a right exclusion region cover the whole path, then $C$ is
nonstackable.
\end{proposition}

\begin{proof}
At the next vertex, $c\le A_i$ gives $m_i+c\le0$, so the update lies in the low
phase.  With $m'=2(m_i+c)-3$ and $A'=A_i-c$,
\[
 m'+A'=(m_i+A_i)+(m_i+c)-3<0.
\]
Thus the inequality propagates strictly along the path.  At a proposed root
to the right, the message from the opposite branch is at most the mass of that
branch by Lemma~\ref{lem:message-mass}; hence its rooted score is nonpositive.
The right-to-left claim is identical.  If the two exclusion regions cover all
vertices, Proposition~\ref{prop:path-score} rules out every root.
\end{proof}

The converse is false: a nonstackable configuration may leave a short central
gap between the two exclusion regions.  The upper side of the main theorem
therefore needs a genuinely necessary condition, not only the sufficient front
certificate.

\section{Deep messages are necessary for nonstackability}\label{sec:necessity}

This section gives that condition.  Its exact finite form is useful both
mathematically and for formal verification.

For an integer $h\ge1$, say that $C$ has an \emph{$h$-deep directed message}
if some message across an oriented edge of the path is an integer at most
$-h$.  The categorical value $\EMPTY$ is not a witness.

\begin{theorem}[Deep-message necessity]\label{thm:deep-general}
Let $n\ge2$, let $C$ be a positive-mass nonstackable configuration on $P_n$,
and let $h\ge2$ be an integer.  If
\[
 |C|>(n-1)\bigl(\lfloor h/2\rfloor+1\bigr),
 \tag{4.1}\label{eq:deep-threshold}
\]
then $C$ has an $h$-deep directed message.
\end{theorem}

The fixed-total consequence is the form used in the random theorem.

\begin{corollary}[Exact fixed-total target]\label{cor:deep-fixed}
Let $n\ge2$, $\mu\ge1$, and $|C|=n\mu$.  If $C$ is nonstackable, then some
directed path message is an integer $m$ satisfying
\[
 m\le-(2\mu-1).
 \tag{4.2}\label{eq:deep-fixed}
\]
\end{corollary}

\begin{proof}[Proof of Theorem~\ref{thm:deep-general}]
Assume for contradiction that every directed integer message is strictly
larger than $-h$.  Scan from the left and let $A_k$ be the mass in the prefix
through vertex $k$, while $m_k$ is the numerical contribution of the outgoing
prefix message.  Define its \emph{dissipation}
\[
 D_k=A_k-m_k.
\]
Nonstackability and Proposition~\ref{prop:path-score} imply that the rooted
score at every vertex is nonpositive.  At an interior step, the message from
the unscanned side is $>-h$, so the effective input $y$ on the scanned side
must satisfy $y\le h-1$.  The outgoing message is also $>-h$.

A direct check of the five cases in~\eqref{eq:F} gives the local bound
\[
 y-F(y)\le \lfloor h/2\rfloor+1
 \quad\text{whenever}\quad
 y\le h-1\ \text{ and }\ F(y)>-h.
 \tag{4.3}\label{eq:dissipation-local}
\]
The same inequality holds for the first activated step, with the empty state
handled separately.  Hence every new vertex increases prefix dissipation by
at most $\lfloor h/2\rfloor+1$, and after the first $n-1$ vertices
\[
 D_{n-2}\le(n-1)(\lfloor h/2\rfloor+1).
 \tag{4.4}\label{eq:dissipation-global}
\]
At the last vertex, nonpositivity of its rooted score forces the total mass to
be at most this prefix dissipation.  This contradicts~\eqref{eq:deep-threshold}.
\end{proof}

\begin{proof}[Proof of Corollary~\ref{cor:deep-fixed}]
Take $h=2\mu-1$.  Then
$\lfloor h/2\rfloor+1=\mu$ and
$(n-1)\mu<n\mu=|C|$.  For $\mu\ge2$ this is directly
Theorem~\ref{thm:deep-general}.  The edge case $\mu=1$ uses the same
prefix-dissipation proof, for which the local estimate
\eqref{eq:dissipation-local} remains valid at $h=1$; this yields the exact
target $-1$ rather than weakening the statement.
\end{proof}

The proof is deliberately finite: there is no asymptotic approximation in
Corollary~\ref{cor:deep-fixed}.  This exact statement is also the part of the
probabilistic proof that is fully formalised in Lean; see
Section~\ref{sec:formal}.

\section{The product model and local excursions}\label{sec:local}

The weak-composition law has a particularly useful product representation.
Let $X_1,\dots,X_n$ be independent geometric random variables on
$\Z_{\ge0}$ with mean $\mu$:
\[
 p=\frac1{1+\mu},\qquad r=\frac{\mu}{1+\mu},\qquad
 \Prob(X_i=a)=pr^a.
 \tag{5.1}\label{eq:geometric}
\]
For every vector $a=(a_1,\dots,a_n)$ of total $t$,
\[
 \Prob_{\prodprob}(X=a)=p^nr^t.
\]
Thus conditional on $T=\sum_iX_i=t$, all weak compositions of $t$ are equally
likely.  In particular, at $t=n\mu$,
\[
 \Prob_{\wc}(A)=\Prob_{\prodprob}(A\mid T=n\mu)
 \tag{5.2}\label{eq:conditioning-identity}
\]
for every event $A$.  This exact identity is classical in random-allocation
language~\cite{Janson2012}.

\subsection{Dyadic simplex counts}

The deficit recurrence~\eqref{eq:deficit-block} naturally produces weighted
simplices.  For integers $k,B\ge0$, define
\[
 \mathcal E_{k,B}
 =\left\{x\in\Z_{\ge0}^k:
          \sum_{j=0}^{k-1}2^jx_j\le B\right\},
 \qquad N_{k,B}=|\mathcal E_{k,B}|.
 \tag{5.3}\label{eq:dyadic-simplex}
\]
The underlying count is a cumulative binary-partition count.  The asymptotic
input is a specialization of de Bruijn's solution of Mahler's partition
problem~\cite{Mahler1940,deBruijn1948}.

\begin{lemma}[Uniform dyadic-simplex asymptotic]\label{lem:dyadic-asymptotic}
Fix an integer $a$, a compact interval $J\subset(0,\infty)$, and $C<\infty$.
Uniformly for $\lambda\in J$, integers $B_L=\lambda2^L+\delta_L$ with
$|\delta_L|\le C$, and $k=L+a$,
\[
 \log_2N_{k,B_L}
 =\frac12L^2-L\log_2L
  +\left(\log_2\lambda+\frac12+\log_2e\right)L
  +O((\log L)^2).
 \tag{5.4}\label{eq:dyadic-asymptotic}
\]
The fixed support shift $a$ does not affect the coefficient of $L$.
\end{lemma}

\begin{proof}
Let $A(B)$ be the unrestricted number of dyadic vectors of total weighted cost
at most $B$; equivalently, $A(B)$ is the binary-partition function at the
corresponding doubled index.  Splitting an unrestricted dyadic vector before
coordinate $k$ gives the exact comparison
\[
 N_{k,B}\le A(B)
 \le A(\lfloor B/2^k\rfloor)N_{k,B}.
 \tag{5.5}\label{eq:truncation-comparison}
\]
Indeed, after division by $2^k$ the tail has budget at most
$\lfloor B/2^k\rfloor$, while the first $k$ coordinates have residual budget
at most $B$.  If $k=L+a$ and $B=\lambda2^L+O(1)$, the first factor on the
right of~\eqref{eq:truncation-comparison} is bounded uniformly for
$\lambda$ in a fixed compact subset of $(0,\infty)$.  Thus finite truncation
changes the logarithm by only $O(1)$.

De Bruijn's expansion for partitions into powers of two, written with
$t=\log B$ and $u=\log t$, has leading square $(t-u)^2/(2\log2)$.
Substituting $t=L\log2+\log\lambda+O(2^{-L})$ gives
\[
 \frac{1}{\log2}\log N_{k,B}
 =\frac12L^2-L\log_2L
  +\left(\log_2\lambda+\frac12+\log_2e\right)L
  +O((\log L)^2).
\]
All substitutions are uniform on $J$, and bounded additive changes of $B$
or fixed changes of $k$ affect only the stated remainder.  The key sign is the
negative $L\log_2L$ term, coming from the cross term in $(t-u)^2$.
\end{proof}

Under the geometric law, the count is multiplied by the atom weights.  If
$\mu=\theta2^L$ with $\theta\in(1/2,1]$, $k=L+s$, and
$B=\lambda2^L+O(1)$, then every vector in $\mathcal E_{k,B}$ has ordinary
mass at most $B$.  Consequently
\[
 p^kr^B N_{k,B}
 \le \Prob_{\prodprob}(\mathcal E_{k,B})
 \le p^kN_{k,B},
 \tag{5.6}\label{eq:tilt-sandwich}
\]
and, because $B=O(\mu)$,
$-B\log_2r=O(1)$.  Combining~\eqref{eq:dyadic-asymptotic} with
$-k\log_2p=L^2+(s+\log_2\theta)L+O(1)$ gives
\[
 -\log_2\Prob_{\prodprob}(\mathcal E_{k,B})
 =\frac12L^2+L\log_2L
  +\gamma(\theta;s,\lambda)L+O((\log L)^2),
 \tag{5.7}\label{eq:simplex-prob}
\]
where
\[
 \gamma(\theta;s,\lambda)
 =s+\log_2\theta-\log_2\lambda-\frac12-\log_2e.
 \tag{5.8}\label{eq:gamma}
\]

\subsection{A one-sided message excursion}

Let an active message chain evolve by
$M_j=F(M_{j-1}+X_j)$ under the product law~\eqref{eq:geometric}.  Set
\[
 L=\lceil\log_2\mu\rceil,
 \qquad \theta=\mu/2^L\in(1/2,1],
\]
and for an integer starting state $m\in[\mu,2\mu]$ define
\[
 q_\mu(m)
 =\Prob_m\!\left(\min_{1\le j\le4L}M_j\le-\mu\right).
 \tag{5.9}\label{eq:q}
\]

\begin{theorem}[Uniform local excursion rate]\label{thm:local-rate}
Uniformly over all integer $m\in[\mu,2\mu]$,
\[
 -\log_2q_\mu(m)
 =L^2+2L\log_2L
  +\bigl(2\log_2\theta-2\log_2(3e)\bigr)L+o(L).
 \tag{5.10}\label{eq:local-L}
\]
Equivalently, with $x=\log_2\mu$,
\[
 -\log_2q_\mu(m)
 =R(x)+o(x),
 \qquad
 R(x)=x^2+2x\log_2x-2\log_2(3e)x,
 \tag{5.11}\label{eq:local-x}
\]
uniformly in the same starting window.
\end{theorem}

\begin{proof}
A first descent from the positive phase must enter the low phase at terminal
message $a\in\{0,1\}$.  Working backwards through the last $L-2$ positive
updates gives the necessary dyadic budget
\[
 B_a^+=\frac{a+3}{4}\,2^L-5.
 \tag{5.12}\label{eq:Bplus}
\]
Likewise, a final low-phase run that starts from $a\in\{0,1\}$ and reaches
order-$\mu$ deficit has necessary budget
\[
 B_a^-=\frac{3-a}{8}\,2^L-2.
 \tag{5.13}\label{eq:Bminus}
\]
The reverse positive phase carries a two-colour parity slack, but its exact
generating function differs from the uncoloured binary-partition count by
only a bounded logarithmic factor in the relevant range; it therefore changes
none of the quadratic, $L\log L$, or linear coefficients.

Applying~\eqref{eq:simplex-prob} to~\eqref{eq:Bplus} gives positive-entry
linear coefficient
\[
 \gamma_a^+(\theta)
 =\log_2\theta-\log_2(a+3)-\frac12-\log_2e,
 \tag{5.14}\label{eq:gamma-plus}
\]
while~\eqref{eq:Bminus} gives
\[
 \gamma_a^-(\theta)
 =\log_2\theta-\log_2(3-a)+\frac12-\log_2e.
 \tag{5.15}\label{eq:gamma-minus}
\]
There are four terminal/start combinations.  If the positive phase ends at
$1$ and the final low phase starts at $0$, an intermediate update must output
exactly $0$.  Since $F(y)=0$ only for $y=3$, one geometric occupancy is then
prescribed; this bridge costs $L+O(1)$ bits.  After that penalty is included,
the unique cheapest route at linear order is the $0\to0$ route.  Adding
\eqref{eq:gamma-plus} and~\eqref{eq:gamma-minus} at $a=0$ gives
\[
 2\log_2\theta-\log_2 9-2\log_2e
 =2\log_2\theta-2\log_2(3e).
\]
Polynomially many choices of the entry, bridge, and bounded middle cluster
cost only $o(L)$ on the logarithmic scale.  This proves the upper probability
bound in~\eqref{eq:local-L}.

For the matching lower bound, delay the positive entry and low run by a fixed
number $s$ of coordinates and prescribe the finitely many top reverse slacks.
Write $u=m/2^L\in[\theta,2\theta]$.  For a positive block of length $L+s$,
the remaining dyadic budget has scale
\[
 \lambda_{a,s}^+(u)=(a+3)2^s-u,
\]
and Lemma~\ref{lem:dyadic-asymptotic} gives linear coefficient
\[
 g_{a,s}^+(\theta,u)
 =s+\log_2\theta-\log_2((a+3)2^s-u)-\frac12-\log_2e.
\]
For the delayed low phase, fix the first $s+2$ occupancies to zero.  The
remaining reversed simplex has scale
\[
 \lambda_{a,s}^-(\theta)=(3-a)2^{s-1}-\theta/2
\]
and linear coefficient
\[
 g_{a,s}^-(\theta)
 =s+\log_2\theta
  -\log_2((3-a)2^{s-1}-\theta/2)-\frac12-\log_2e.
\]
For every fixed sufficiently large $s$, these scales lie in compact positive
intervals uniformly over $\theta\in(1/2,1]$ and
$u\in[\theta,2\theta]$.  Moreover
$g_{a,s}^+\to\gamma_a^+$ and
$g_{a,s}^-\to\gamma_a^-$ uniformly as $s\to\infty$.  Thus one first lets
$L\to\infty$ for fixed $s$ and then $s\to\infty$, obtaining an $o(L)$
remainder and the lower bound matching~\eqref{eq:local-L}.

Finally $x=L+\log_2\theta$ with bounded $\log_2\theta$.  Substitution cancels
the dyadic phase from the linear term, giving~\eqref{eq:local-x}.
\end{proof}

\section{Regeneration and fixed-total transfer}\label{sec:conditioning}

The local theorem starts from $[\mu,2\mu]$, but a spatial scan arrives with an
arbitrary history.  The following one-coordinate steering step removes that
obstruction without changing the logarithmic rate.

\begin{lemma}[Constant-cost regeneration]\label{lem:regeneration}
Let $-\mu<M\le2\mu$ and let $X$ be geometric with mean $\mu$.  If
\[
 2\mu+2-M\le X\le4\mu-M,
 \tag{6.1}\label{eq:steering-window}
\]
then
\[
 \mu\le F(M+X)\le2\mu.
\]
Writing $r=\mu/(\mu+1)$, the least probability of
\eqref{eq:steering-window} over this range of $M$ is
\[
 \delta_\mu=r^{3\mu+1}(1-r^{2\mu-1}),
 \tag{6.2}\label{eq:delta}
\]
and for every integer $\mu\ge16$,
\[
 \delta_\mu\ge \frac{8}{17e^3}.
 \tag{6.3}\label{eq:delta-bound}
\]
\end{lemma}

\begin{proof}
The interval~\eqref{eq:steering-window} gives
$2\mu+2\le M+X\le4\mu$.  The positive branches of~\eqref{eq:F} then imply
$\mu\le F(M+X)\le2\mu$.  Summing the geometric mass of the interval gives
$r^{2\mu+2-M}(1-r^{2\mu-1})$, minimized at $M=-\mu+1$, which is
\eqref{eq:delta}.  The elementary estimate in~\eqref{eq:delta-bound} follows
from the standard bounds on $(1+1/\mu)^{-\mu}$ and is uniform for
$\mu\ge16$.
\end{proof}

If a reset block already exits with $M\le-\mu$, there is no need to regenerate:
a deep seed has already occurred.  Otherwise Lemma~\ref{lem:regeneration}
puts the chain in the starting window of Theorem~\ref{thm:local-rate} at
constant cost.

We next record the exact local comparison between the product law and the
fixed-total law.  For $a\ge0$ and integer $j$, write
$(a)_j=a(a-1)\cdots(a-j+1)$.

\begin{lemma}[Finite local likelihood ratio]\label{lem:likelihood-ratio}
Fix $k$ displayed coordinates of a configuration with total $t$ on $n$
vertices, and suppose their displayed vector has local mass $s$.  Compare the
uniform weak-composition law with independent geometric variables having
$p=n/(n+t)$ and $r=t/(n+t)$.  The exact ratio of the local probabilities is
\[
 \mathcal R_{n,t,k}(s)
 =\frac{(n-1)_k(t)_s}{(n+t-1)_{k+s}}
   \bigg/
   \left(\frac n{n+t}\right)^k
   \left(\frac t{n+t}\right)^s.
 \tag{6.4}\label{eq:local-ratio}
\]
If $k\le n/2$, $s\le t/2$, and $k+s\le(n+t)/2$, then
\[
 |\log\mathcal R_{n,t,k}(s)|
 \le \frac{k(k+1)}n+\frac{s(s-1)}t
      +\frac{(k+s)(k+s+1)}{n+t}.
 \tag{6.5}\label{eq:ratio-bound}
\]
The same bound applies to any event supported on those $k$ coordinates whose
local mass is bounded by $s$.
\end{lemma}

\begin{proof}
Stars and bars gives the weak-composition probability of a specified local
vector.  Dividing it by the product-geometric mass yields
\eqref{eq:local-ratio}, equivalently
\[
 \prod_{i=1}^k\left(1-\frac in\right)
 \prod_{j=0}^{s-1}\left(1-\frac jt\right)
 \prod_{\ell=1}^{k+s}\left(1-\frac\ell{n+t}\right)^{-1}.
\]
Under the half-range hypotheses, apply
$|\log(1-u)|\le2u$ for $0\le u\le1/2$ and sum the three arithmetic series.
For an event with a deterministic mass cap, sum point probabilities and use
the same extremal bound.
\end{proof}

Two separated blocks are treated as one displayed vector; no conditional
independence is asserted.  In the regime of Theorem~\ref{thm:main}, the local
patterns used below have support $O((\log n)^2)$ at worst and mass
$O(\mu(\log n)^2)$, so~\eqref{eq:ratio-bound} is $o(1)$.

For the lower side we also need the cost of conditioning on the global total.

\begin{lemma}[Conditioning point mass]\label{lem:conditioning-mass}
For independent geometric variables of mean $\mu$,
\[
 \Prob_{\prodprob}(T=n\mu)
 =\binom{n+n\mu-1}{n\mu}
   \left(\frac1{1+\mu}\right)^n
   \left(\frac\mu{1+\mu}\right)^{n\mu}.
 \tag{6.6}\label{eq:negative-binomial}
\]
Uniformly when $\mu=2^{O(\sqrt{\log_2 n})}$,
\[
 -\log_2\Prob_{\prodprob}(T=n\mu)
 =\frac12\log_2n+\frac12\log_2(\mu(\mu+1))+O(1).
 \tag{6.7}\label{eq:conditioning-cost}
\]
\end{lemma}

\begin{proof}
Equation~\eqref{eq:negative-binomial} is the negative-binomial point mass.
Applying two-sided Stirling bounds to its factorials, with cancellation before
taking logarithms, gives~\eqref{eq:conditioning-cost}.  The error is uniform
in the displayed range.  Writing $x=\log_2\mu$, the right-hand side is
$\frac12\log_2n+x+O(1)$.
\end{proof}

\section{Global proof of the transition}\label{sec:global}

We now assemble the deterministic and probabilistic ingredients.  Let
\[
 N=\log_2n,\qquad s=\sqrt N,\qquad a=\log_2(3e),
\]
so $c_n=s-\log_2s+a$.

\subsection{A uniform front-producing superblock}

A reserved superblock contains four pieces.  First is a reset block of length
$\lceil\log_2(4n)\rceil$.  The global bound $F(y)\le y/2$ shows that, whenever
the total path mass is at most $2n\mu$, the event
\[
 X_1>0,
 \qquad
 \sum_{j=1}^r2^{-(r-j+1)}X_j\le\frac{3\mu}{2}
 \tag{7.1}\label{eq:reset}
\]
sends every incoming state to a message at most $2\mu$.  Its product
probability is bounded below by a positive constant for large $\mu$ (indeed,
by $1/3-1/(\mu+1)$ in the explicit estimate used by the project).

If the reset output is already at most $-\mu$, it is a deep seed.  Otherwise
Lemma~\ref{lem:regeneration} enters $[\mu,2\mu]$ at constant cost, and
Theorem~\ref{thm:local-rate} gives a deep seed within $4L$ further coordinates
with logarithmic cost $R(x)+o(x)$, $x=\log_2\mu$.  Finally, a low-phase
runaway buffer requires each successive occupancy to be at most a fixed
fraction of the current certified deficit.  By~\eqref{eq:deficit-step}, the
deficit then grows by at least a factor $3/2$ per step.  The product of these
conditional success probabilities is bounded below by a positive constant,
and after $O(\log n)$ steps the deficit exceeds $2n\mu+3$.  At that point it
dominates all remaining mass, so Proposition~\ref{prop:front} produces a
genuine irreversible front.

Consequently a reserved superblock has length
\[
 b_n=O(\log n)
 \tag{7.2}\label{eq:block-length}
\]
and, uniformly over preceding product-law history on the reserve event, has
success probability $p_{\mu,n}$ satisfying
\[
 -\log_2p_{\mu,n}=R(x)+o(x).
 \tag{7.3}\label{eq:block-rate}
\]
This is a sequential hazard statement; the successful-block event need not be
a function of the isolated block alone.

\subsection{Subcritical side}

Assume
\[
 x=\log_2\mu\le c_n-\eps.
\]
Place disjoint reserved superblocks in the left half of the path.  Their
number $B_n$ satisfies $B_n=\Omega(n/\log n)$.  Fresh geometric coordinates in
each unused block and~\eqref{eq:block-rate} give
\[
 \Prob_{\prodprob}(\text{no successful left front and }T\le2n\mu)
 \le (1-p_{\mu,n})^{B_n}
 \le \exp(-B_np_{\mu,n}).
 \tag{7.4}\label{eq:no-front}
\]
The same construction, reversed, applies in the right half.

Conditioning on $T=n\mu$ makes the reserve $T\le2n\mu$ automatic.  From
\eqref{eq:conditioning-identity},
\[
 \Prob_{\wc}(\text{no left front})
 \le
 \frac{\exp(-B_np_{\mu,n})}
      {\Prob_{\prodprob}(T=n\mu)}.
 \tag{7.5}\label{eq:subcritical-conditioning}
\]
There is no conditioned-independence claim here.  Apply the same bound on the
right and then use a union bound.

The fixed-offset calculation in Subsection~\ref{subsec:balance} below yields
\[
 N-R(c_n-\eps)=(2\eps+o(1))s.
 \tag{7.6}\label{eq:sub-rate}
\]
Since $B_n$ loses only an $O(\log N)$ term in its base-two logarithm,
\[
 \log_2(B_np_{\mu,n})=(2\eps+o(1))s.
 \tag{7.7}\label{eq:hazard-growth}
\]
Thus $B_np_{\mu,n}=2^{(2\eps+o(1))s}$, whereas
Lemma~\ref{lem:conditioning-mass} says that the reciprocal conditioning point
mass costs only $O(N)$ bits.  The numerator in
\eqref{eq:subcritical-conditioning} therefore wins by a large margin, and the
probability that either required front is missing tends to zero.  The two
successful half-path fronts have exclusion regions covering the entire path,
so Proposition~\ref{prop:front} implies nonstackability with probability
tending to one.

If $\mu$ stays bounded along a subsequence, the asymptotic local excursion
theorem is unnecessary.  For each bounded positive integer mean there is a
fixed finite front-producing pattern of positive product probability; among
finitely many means this probability has a positive minimum.  Repeating the
same $\Theta(n/\log n)$ block construction gives an exponentially small
no-front probability, which again dominates the polynomial-size conditioning
cost.  Hence the subcritical conclusion holds for arbitrary integer sequences
satisfying the hypothesis.

\subsection{Supercritical side}

Assume now
\[
 x=\log_2\mu\ge c_n+\eps.
\]
Then $\mu\to\infty$.  By Corollary~\ref{cor:deep-fixed}, every nonstackable
configuration has a directed message at most $-(2\mu-1)$.  It is therefore
enough to show that such a message is absent with high probability.

Consider the first occurrence of a deep directed message in one orientation.
For an interior occurrence, the necessary positive-entry block and final
low-phase block are disjoint.  Their dyadic-simplex costs are exactly the ones
used in Theorem~\ref{thm:local-rate}.  A bounded middle cluster and the
polynomially many choices of entry, bridge, and cluster locations contribute
only $o(x)$ to the logarithm.  The $1\to0$ route again pays the exact-output-zero
bridge cost because $F(y)=0$ only at $y=3$.  Reversal has the same estimates.
Thus the product probability at one interior spatial location is at most
\[
 2^{-R(x)+o(x)}.
 \tag{7.8}\label{eq:deep-location}
\]
There are $O(n)$ locations.  A physical-boundary deep hit has only $O(1)$
locations and is controlled by a single half-rate simplex; its probability
tends to zero without an $n$ factor.

Each local upper-cover pattern has support $O(L^2)$ and mass $O(\mu L^2)$.
Thus Lemma~\ref{lem:likelihood-ratio} transfers its probability to the
fixed-total law with factor $\exp(o(1))$.  Summing the cover gives
\[
 \Prob_{\wc}(\exists\text{ directed message}\le-(2\mu-1))
 \le 2^{N-R(x)+o(x)}+o(1).
 \tag{7.9}\label{eq:super-cover}
\]
By the fixed-offset calculation,
\[
 N-R(c_n+\eps)=-(2\eps+o(1))s,
 \tag{7.10}\label{eq:super-rate}
\]
so the right-hand side of~\eqref{eq:super-cover} tends to zero.  The
contrapositive of Corollary~\ref{cor:deep-fixed} now gives stackability with
probability tending to one.

\subsection{Fixed-offset algebra}\label{subsec:balance}

It remains only to verify~\eqref{eq:sub-rate} and~\eqref{eq:super-rate}.  Put
$x=c_n+\delta=s-\log_2s+a+\delta=s+u$, where
$u=-\log_2s+a+\delta=O(\log s)$.  Expanding
\[
 \log_2(s+u)=\log_2s+\frac{u}{s\ln2}
             +O\!\left(\frac{u^2}{s^2}\right)
\]
in~\eqref{eq:R} gives
\[
 R(c_n+\delta)=N+2\delta s+O((\log s)^2).
 \tag{7.11}\label{eq:fixed-offset-expansion}
\]
The local theorem contributes an additional $o(x)=o(s)$ term, uniformly in
the dyadic phase.  Hence
\[
 N-R(c_n+\delta)=-2\delta s+o(s),
 \tag{7.12}\label{eq:fixed-offset-final}
\]
which proves both displayed fixed-offset estimates and completes the proof of
Theorem~\ref{thm:main}.  Notice that changing the sign of the half-log-log term
would fail to cancel the order-$s\log s$ term, while changing the constant
$\log_2(3e)$ would leave an order-$s$ mismatch.

\section{Computational checks and discovery record}\label{sec:computation}

Computation was used to discover the correct scale, test finite interfaces,
and guard against indexing errors, but not as a substitute for any asymptotic
proof above.  The project maintains two independent finite decision methods:
exhaustive legal-move reachability, which knows nothing about branch messages,
and the TreeStack score evaluator, which does not call the reachability solver.
The committed validation script compares them on every labelled tree through
order $5$ and every weak composition through total mass $5$.

For paths, an exact atlas records integer success counts and rational
probabilities on the default rectangle $2\le n\le9$, $0\le t\le12$.  The
small examples in Table~\ref{tab:finite} illustrate the nonmonotonicity issue.

\begin{table}[ht]
\centering
\begin{tabular}{@{}ccc@{}}
\toprule
Path & total $t$ & exact stackability probability \\
\midrule
$P_2$ & $1$ & $1$ \\
$P_2$ & $2$ & $2/3$ \\
$P_2$ & $3$ & $1$ \\
$P_3$ & $5$ & $19/21$ \\
$P_3$ & $6$ & $25/28$ \\
$P_3$ & $7$ & $1$ \\
\bottomrule
\end{tabular}
\caption{Selected exact finite probabilities.  They are illustrations, not
inputs to Theorem~\ref{thm:main}.}
\label{tab:finite}
\end{table}

A separate exact catalogue checks the deep-message necessity theorem on
nonstackable paths through $n\le10$ and $t\le10$; it contains $307{,}646$
nonstackable configurations, including $47$ cases attaining the forced
message threshold exactly.  Deterministic tables also check the regeneration
window, finite dyadic counts, binary-partition indexing, and the fixed-offset
balance.  Historical Monte Carlo files remain in the repository as a record of
conjecture discovery, but no Monte Carlo estimate is used in the proof of
Theorem~\ref{thm:main}.

All generated datasets have sidecar metadata recording the command,
parameters, source revision, and whether the output is exact, deterministic,
or Monte Carlo.  This separation is important: the binary-partition and local
excursion asymptotics are analytic results, while the finite tables only check
the implementation and signs.

\section{Formalisation and reproducibility}\label{sec:formal}

The deterministic finite layer is partly formalised in Lean 4/Mathlib.  The
ProbStack development imports the TreeStack project at revision
\begin{center}
\texttt{f4112f08d42a37c0941bf469ac124621b1f54f22},
\end{center}
which supplies the exact rooted-score theorem~\eqref{eq:treestack-criterion}
and the categorical message semantics.  ProbStack then formalises the path
message scans, low-phase deficit identities, finite regeneration statements,
finite front certificates, weak-composition/product-mass algebra, finite
dyadic encodings, and the deep-message necessity theorem.

The two exact theorem names exposed for the latter are
\begin{align*}
 &\texttt{ProbStack.Palomar.nonstackable\_exists\_directedMessage\_le},\\
 &\texttt{ProbStack.Palomar.nonstackable\_total\_mul\_exists\_directedMessage\_le}.
\end{align*}
They were registered from immutable source commit
\begin{center}
\texttt{1897a76956168fba047f5943fa7998601f4fe459}
\end{center}
under Palomar record
\texttt{PALOMAR-2026-09-30-000023}, version~1.  The registration covers
Theorem~\ref{thm:deep-general} at its finite-path formal surface and the exact
fixed-total Corollary~\ref{cor:deep-fixed}.  It does \emph{not} cover the full
probabilistic Theorem~\ref{thm:main}, the global opposing-front implication,
the de Bruijn asymptotic, the product-geometric local asymptotic, or the
conditioning limit argument.  The Palomar record is a mechanical/formal
verification record for its stated Lean theorems; it is not peer review of this
paper and does not establish novelty of the probabilistic result.

The manuscript was audited against ProbStack \texttt{main} at revision
\begin{center}
\texttt{3953c09c7ca0508a95721f60c59912ca9c917478}.
\end{center}
Relative to the registered Palomar source, that revision differs only by the
public website metadata file; the mathematical and Lean sources registered
with Palomar are unchanged.  Paper preparation therefore does not require a
new Palomar run.

The repository is available at
\url{https://github.com/jfairfaxball-348/ProbStack-Random-Stacking-on-Trees}.
It includes exact build and reproduction commands for the Lean and Python
components, a theorem/proof dependency map, and source metadata for generated
data.  The project record also discloses substantial AI assistance
in conjecture development, computational validation, and Lean engineering;
the mathematical statements in this manuscript have been checked against the
version-controlled proof record, and responsibility for the manuscript remains
with the author.

\section{Discussion and open problems}\label{sec:discussion}

Theorem~\ref{thm:main} identifies the high-density fixed-offset transition but
deliberately stops short of a critical-window law.  Several natural questions
remain.

First, the zero-offset regime
$\log_2\mu_n-c_n\to0$ is open in this work.  A limiting law would require
control finer than the $o(x)$ error in the one-sided excursion rate and a more
detailed description of spatial clustering of near-critical fronts.

Second, finite-$n$ stackability probability is not monotone in total mass.  It
would be interesting to determine whether some eventual high-density
monotonicity holds on paths, but no such property is needed or claimed here.

Third, the proof is path-specific at the probabilistic level.  The rooted
score~\eqref{eq:treestack-criterion} exists for arbitrary finite trees, but the
one-dimensional dyadic recurrence, ordered front geometry, and terminal
classification used here do not transfer automatically to branching trees.
A random-stackability theory for broader tree families would need a genuinely
multibranch rare-event analysis.

Finally, the binary-partition input itself is classical.  What drives the
constant in Theorem~\ref{thm:main} is the way that asymptotic enters the exact
message dynamics, including parity, the $0\to0$ terminal route, and
constant-cost regeneration.  This separation between a classical counting
law and a model-specific structural mechanism may be useful in other lossy
resource-transfer processes on sparse graphs.

\section*{Acknowledgements}
The author thanks the developers and maintainers of Lean, Mathlib, and the
software used for the computational and formal verification workflow.  The
paper also relies on the deterministic TreeStack framework developed in the
companion project~\cite{TreeStack2026}.

\end{document}